\documentclass[10pt,a4paper]{article}
\usepackage[utf8]{inputenc}

\usepackage{amsmath}
\usepackage{amsfonts}
\usepackage{amssymb}
\usepackage{graphicx}
\usepackage{mathtools}
\usepackage{setspace}

\usepackage[square, comma, numbers]{natbib}
\usepackage{amsthm}
\theoremstyle{plain}
\usepackage{appendix}
\usepackage{tikz}
\usetikzlibrary{cd}

\usepackage{hyperref}

\theoremstyle{plain}
\newtheorem{thm}{Theorem}
\newtheorem{definition}[thm]{Definition}

\newtheorem{lem}[thm]{Lemma}

\newtheorem{prop}[thm]{Proposition}

\theoremstyle{definition}
\newtheorem{ex}[thm]{Example}

\newtheorem{rem}[thm]{Remark}

\newcommand{\C}{\mathbb{C}}
\newcommand{\Z}{\mathbb{Z}}

\newcommand{\cA}{\mathcal{A}}

\newcommand{\Oh}{\mathcal{O}}

\newcommand{\im}{\operatorname{im}}
\newcommand{\del}{\partial}
\newcommand{\delbar}{\overline{\partial}}

\newcommand{\pr}{\operatorname{pr}}

\newcommand{\Pro}{\mathbb{P}}

\mathtoolsset{centercolon=true}
\newcommand{\Cdot}{{\raisebox{-0.7ex}[0pt][0pt]{\scalebox{2.0}{$\cdot$}}}}
\mathchardef\mhyphen="2D

\let\oldabstract\abstract
\let\oldendabstract\endabstract
\makeatletter
\renewenvironment{abstract}
{%
	{\list{}{\addtolength{\leftmargin}{3em} 
			\listparindent 0em%
			\itemindent    \listparindent%
			\rightmargin   \leftmargin%
			\parsep        \z@ \@plus\p@}%
		\item\relax}%
	{\endlist}%
	\oldabstract}
{\oldendabstract}
\makeatother

\usepackage{authblk}
\title{The Double Complex of a Blow-up}
\author{Jonas Stelzig\thanks{\textit{email:} jonas.stelzig@wwu.de}\\
	WWU Münster}
\begin{document}

\maketitle 
\begin{abstract}
We compute the double complex of smooth complex-valued differential forms on projective bundles over and blow-ups of compact complex manifolds up to a suitable notion of quasi-isomorphism. This simultaneously yields formulas for ``all'' cohomologies naturally associated with this complex (in particular, de Rham, Dolbeault, Bott-Chern and Aeppli).
\end{abstract}

\section{Introduction}
Given a compact complex manifold $X$ and a complex submanifold $Z\subseteq X$ of codimension $r$ at least $2$, the blow-up $\widetilde{X}$ of $X$ along $Z$ is a new complex manifold, roughly obtained from $X$ by replacing $Z$ with the space of all directions into $Z$, i.e. the projectivized normal bundle. A natural task is to express cohomological invariants of $\widetilde{X}$ in terms of those of $X$ and $Z$. In deliberately vague notation, the expected (additive) relation is the following:
\[
H(\widetilde{X})\cong H(X)\oplus \bigoplus_{i=1}^{r-1}H(Z)[i],
\]
where $H$ means some cohomology and $[i]$ denotes an appropriate degree shift.\\

In this article, we compute the double complex of complex-valued forms $\cA_{\widetilde X}$ for $\widetilde{X}$ a blow-up up to ``$E_1$-isomorphism''. This establishes the above formula for all linear functors (from the category of double complexes to, say, vector spaces) that map $E_1$-isomorphisms to isomorphisms. In particular, one obtains a uniform proof for Dolbeault cohomology, the higher pages of the Fr\"olicher spectral sequence, the de Rham cohomology and the Bott-Chern and Aeppli cohomologies. As intermediate steps that might be of independent interest, we also compute the double complex of the projectivization $\Pro(E)$ of a vector bundle $E$ of rank $n$ on $X$, resulting in a formula for cohomology of the type
\[
H(\Pro(E))\cong\bigoplus_{i=0}^{n-1}H(Z)[i],
\] 
and show that for a modification $\widetilde{X}\longrightarrow X$, the cohomology of $\widetilde{X}$ contains that of $X$ as a direct summand.\\

The key algebraic point of our argument is to stay on the level of (double) complexes and avoid passing to cohomology as long as possible. A main geometric input is a formula from \cite{guillen_critere_2002} for the higher direct image sheaves of the holomorphic differentials on the blow-up.\\

There is a wide range of related work that present article fits into, mostly focussing on particular cohomology theories: The blow-up formula for de Rham cohomology has been established in \cite[p. 605 f.]{griffiths_principles_1978}, using the Thom-isomorphism. In the case when $X$ is K\"ahler, this is refined to include the Hodge structure on the de Rham cohomology, i.e. Dolbeault-cohomology, in \cite[7.3.3]{voisin_hodge_2002}. More recently, there has been a lot of activity on extending these results to cohomologies other than de Rham in the non-K\"ahler case: In \cite{rao_dolbeault_2019}, the existence of an isomorphism for the Dolbeault cohomology of general (compact) complex manifolds is established and the formula for the Bott-Chern cohomology is conjectured. The isomorphism in the Dolbeault case is made explicit in \cite{meng_explicit_2018} and the compactness hypothesis is removed. The method used there is conceptualized and extended in \cite{meng_mayer-vietoris_2018}. The Bott-Chern case is partially proved in \cite{yang_bott-chern_2017} and complemented with a conjecture for the Bott-Chern cohomology of projective bundles (c.f. also \cite{rao_dolbeault_2019}) which implies the full formula in this case. Independently, the Dolbeault case is also established in \cite{angella_note_2017} in the situation that $Z$ admits a neighborhood with a holomorphic retract to $Z$, using a Thom-isomorphism for Dolbeault cohomology. In \cite{meng_morse-novikov_2018}, the case of Morse-Novikov cohomology is considered and in \cite{rao_dolbeault_2018}, an explicit isomorphism is given for Dolbeault cohomology with values in a vector bundle.\\

The present article has mostly been written after the first preprint versions of \cite{rao_dolbeault_2019}, \cite{yang_bott-chern_2017} and \cite{angella_note_2017} had appeared and the author is glad to have been influenced by them. Vice versa, it is hopefully fair to say that the preprint version of the present article has had some effect on parts of the latest versions of several of the more recent works on the topic, see e.g. \cite[p. 4f and Prop. 3.4]{rao_dolbeault_2019}, \cite[Lem. 4.1 and App. B]{rao_dolbeault_2018}, \cite[Lem. 4.1]{meng_explicit_2018}, \cite[App. A.2]{yang_bott-chern_2017}.\\

Given the central position of the blow-up construction in complex geometry and, more specifically, bimeromorphic geometry, results as the ones proven here can be put to use in many ways. We refer the reader to all the articles mentioned above for a plethora of beautiful applications. In the present work, we content ourselves with some example calculations as its further consequences are more naturally derived after a more in-depth discussion of the notion of $E_1$-isomorphism, which is given in \cite{stelzig_structure_2018}.

\section{Preliminaries}\label{sec: Preliminaries}
It will be convenient to abstract the algebraic properties of the double complex of $\C$-valued smooth differential forms on a complex manifold.\\

We will be considering bounded double complexes over the complex numbers with real structure, i.e. quadruples
\[
(A^{\Cdot,\Cdot},\del_1,\del_2,\sigma)
\]
consisting of
\begin{itemize}
	\item a (not necessarily finite dimensional) $\Z^2$-graded $\C$-vector space $A^{\Cdot,\Cdot}$, s.t. $A^{p,q}=0$ for almost all $(p,q)\in\Z^2$,
	\item two $\C$-linear maps $\del_1$ and $\del_2$ of degrees $(1,0)$ and $(0,1)$ which satisfy $\del_1^2=\del_2^2=0$ and $\del_1\del_2+\del_2\del_1=0$.
	\item a conjugation-antilinear involution $\sigma$ on $A^{\Cdot,\Cdot}$, satisfying $\sigma A^{p,q}=A^{q,p}$ and $\sigma\del_1\sigma=\del_2$.
\end{itemize} 

To ease notation and language, in the following, we will say double complex instead of bounded double complex over the complex numbers with real structure and write $A$ instead of $(A^{\Cdot,\Cdot},\del_1,\del_2,\sigma)$. By a map of double complexes, we mean a $\C$-linear map of the underlying vector spaces, compatible with the bigrading, differentials and real structure.

\begin{ex} Let $X$ be a connected compact complex manifold of dimension $n$.
	\begin{itemize}
		\item The \textbf{Dolbeault double complex} $\cA_X=(\cA_X^{\Cdot,\Cdot},\del,\delbar,\sigma)$ of $\C$-valued smooth differential forms on $X$. 
		\item The \textbf{double complex of currents} or dual Dolbeault double complex $\mathcal{D}^{top}\cA_X$, consisting in degree $(p,q)$ of the topological dual of $\cA_X^{n-p,n-q}$ (i.e. ``currents'', see \cite[par. 8.,9.]{serre_theoreme_1955} for more details), with differentials \begin{align*}		
		\del_{\mathcal{D}^{top}\cA_X}^{p,q}:=(\varphi\mapsto (-1)^{p+q+1}\varphi\circ\del^{n-p-1,n-q})\\
		\overline{\del}_{\mathcal{D}^{top}\cA_X}^{p,q}:=(\varphi\mapsto (-1)^{p+q+1}\varphi\circ\overline{\del}^{n-p,n-q-1})
		\end{align*} 
		\item For any double complex $A$ and integer $i\in\Z$, there is the \textbf{shifted double complex} $A[i]$ with same differentials and involution but new bigrading $(A[i])^{p,q}:=A^{p-i,q-i}$.
	\end{itemize}
\end{ex}
Associated with any double complex $A$ are several cohomologies: 
\begin{itemize}
	\item de Rham (or ``total'') cohomology: $H_{dR}^k(A):=H^k(\bigoplus_{p+q=\Cdot}A^{p,q},\del_1+\del_2)$
	\item Dolbeault (or ``column'') cohomology: $H_{\del_2}^{p,q}(A):=H^q(A^{p,\Cdot},\del_2)$
	\item conjugate Dolbeault (or ``row'') cohomology: $H_{\del_1}^{p,q}(A):=H^p(A^{\Cdot,q},\del_1)$
	\item Bott-Chern cohomology:  $H^{p,q}_{BC}(A):=\big(\frac{\ker\del_1\cap\ker\del_2}{\im\del_1\circ\del_2}\big)^{p,q}$
	\item Aeppli cohomology:
	$H^{p,q}_{A}(A):=\big(\frac{\ker\del_1\circ\del_2}{\im\del_1+\im\del_2}\big)^{p,q}$
	\end{itemize}
There is the ``Fr\"olicher spectral sequence'', converging from Dolbeault to de Rham cohomology:
\[
FS : E_1^{p,q}=H^{p,q}_{\del_2}(A)\Longrightarrow H_{dR}^{p+q}(A).
\]
and an analoguous spectral sequence starting from conjugate Dolbeault. Between the other cohomologies, there are maps induced by the identity. The whole situation is summarized in the following diagram (c.f \cite{angella_cohomological_2013}):
	\[\begin{tikzcd}
	&H_{BC}^{p,q}(A)\ar[ld]\ar[d]\ar[rd]&\\
	H_{\del_1}^{p,q}(A)\ar[rd]\arrow[Rightarrow]{r}&(H_{dR}^{p+q}(A),F_1,F_2)\ar[d]&H^{p,q}_{\del_2}(A)\ar[ld]\arrow[Rightarrow]{l}\\
	&H_A^{p,q}(A)
	\end{tikzcd}\]
\begin{definition}
	A morphism of double complexes is called an \textbf{$E_1$-isomorphism}, if it induces an isomorphism in Dolbeault cohomology. 
\end{definition}
\begin{ex}
	For any connected compact complex manifold, the map 
	\begin{align*}
	\Phi:\cA_X&\longrightarrow \mathcal{D}^{top}\cA_X\\
	\omega&\longmapsto \int_X\omega\wedge\_
	\end{align*}
	is an $E_1$-isomorphism by Serre duality (\cite{serre_theoreme_1955})
\end{ex}

It is well-known that an $E_1$-isomorphism automatically induces an isomorphism on all later pages of the Fr\"olicher spectral sequence (and its conjugate) and in de Rham cohomology. Further, one has the following Lemma (see \cite[Thm. 2.7]{angella_cohomologies_2013-1} and \cite[Thm. 1.3]{angella_cohomologies_2017} for special cases and \cite{stelzig_structure_2018} for a proof in the general
setting and further investigation on the notion of $E_1$-isomorphism):

\begin{lem}
	Any $E_1$-isomorphism induces an isomorphism in Bott-Chern and Aeppli cohomology.
\end{lem}

\section{Projective bundles}
The Dolbeault case of the following proposition is proved in \cite{rao_dolbeault_2019}. We show here that one can define all relevant maps on the double complex level.
\begin{prop}\label{complex of a projective bundle}
	Let $\widetilde{\pi}:E\longrightarrow X$ be a complex vector bundle of rank $n$ over a complex manifold $X$ and $\pi:\Pro(E)\longrightarrow X$ the associated projective bundle. Denoting by $K$ the double complex $K:=\bigoplus_{i=0}^{n-1}\cA_X[i]$, there is a commutative diagram
	\[
	\begin{tikzcd}
	&\cA_X\ar[ld]\ar[d]\ar[rd,"\pi^*"]&\\
	\cA_X\otimes\cA_{\Pro^{n-1}}&K\ar[l]\ar[r]&\cA_{\Pro(E)},
	\end{tikzcd}
	\]
	such that the horizontal maps are $E_1$-isomorphisms and the others induce injective maps in Dolbeault-cohomology.
\end{prop}

\begin{proof}
	Let 
	\[
	T:=\{(e,p)\in E\times\Pro(E)\mid e\in p\}\subseteq \pi^*E
	\]
	denote the tautological bundle on $\Pro(E)$. For any fibre $F_x:=\pi^{-1}(x)\cong\Pro^{n-1}$, there is an identification $T|_{F_x}\cong \Oh_{\Pro^{n-1}}(-1)$. Choose some hermitian metric $g$ on $T$ and let $\theta\in\cA^2_{\Pro(E)}$ be the curvature of the Chern connection defined by $g$, s.t. 
	\[
	c_1(T)=\left[\frac{1}{2\pi i}\theta\right]\in H^2_{\delbar}(\Pro(E)).
	\] 
	It is known that $\theta$ is a closed $(1,1)$-form and because $0\neq c_1(\Oh_{\Pro^{n-1}}(-1))=c_1(T)|_{F_x}$, it is not exact. Denote $\theta_x:=\theta|_{F_x}$ and  by $\cA(\theta)$, resp. $\cA(\theta_x)$ the finite dimensional (as $\C$-vector space) subcomplexes of $\cA_{\Pro(E)}$, resp. $\cA_{F_x}$ with basis $\{1,\theta,\theta^2,...,\theta^{n-1}\}$  resp. $\{1,\theta_x,\theta_x^2,...,\theta_x^{n-1}\}$. With this, we can redefine $K$ as $K:=\cA_X\otimes\cA(\theta)$, which is equivalent to the definition given in the statement. The bigraded Dolbeault cohomology algebra of $\Pro^{n-1}$ is given by $H_{\delbar}^{\Cdot,\Cdot}(\Pro^{n-1})=\C[t]/(t^{n-1})$ with $t=c_1(\Oh_{\Pro^{n-1}}(-1))$ of bidegree $(1,1)$. In particular, restriction and projection to cohomology (all forms in $\cA(\theta_x)$ are closed) yield isomorphisms of double complexes $\cA(\theta)\cong\cA(\theta_x)\cong H(\Pro^{n-1})$ and the inclusion
	\[
	\cA(\theta_x)\longrightarrow\cA_{F_x}
	\]
	is an isomorphism on the first page of the Fr\"olicher spectral sequence. Thus, we can define the left hand map in the diagram in the statement as the composite
	\[
	\cA_X\otimes\cA(\theta)\longrightarrow\cA_X\otimes\cA(\theta_x)\longrightarrow\cA_X\otimes\cA_{\Pro^{n-1}},
	\]
	where the maps are the identity on the first factor and restriction and inclusion on the second factor.\\
	
	The right hand map in the diagram in the statement is given by $\pi^*$ on the first factor and the inclusion on the second. It is an isomorphism on the first page of the Fr\"olicher spectral sequence by the Hirsch Lemma for Dolbeault cohomology, proven in \cite[Lem. 18]{cordero_compact_2000}, as a consequence of a spectral sequence introduced by Borel in the appendix to \cite{hirzebruch_topological_1978}.
	The left diagonal and the vertical map are inclusions to the first factor of the tensor product and commutativity is clear by definition.
\end{proof}

\begin{rem}
	Since the maps are defined on the level of complexes, this result allows in particular the computation not only of the Dolbeault, but also of the Bott-Chern and Aeppli cohomologies of a projective bundle, thereby confirming the formula conjectured in \cite{rao_dolbeault_2019}.\\
\end{rem}
\begin{rem} (Products)
	Since we get that $H_{BC}(\Pro(E))\cong\bigoplus_{i=0}^{n-1}H_{BC}(X)\wedge[\theta^i]$, there is a $\del\delbar$-exact form $\eta\in \cA_{\Pro(E)}^{n,n}$ and $d$-closed forms $c_i\in \cA_X^{i,i}$ s.t. $\theta^n+\pi^*c_1\theta^{n-1}+...+\pi^*c_{n-1}\theta+c_n=\eta$. Denote the left hand side of this equation by $P(\theta)$. We may then redefine $K:=\cA_X[\theta]/(P(\theta))$, which clearly has the same additive structure as before but is also equipped with a multiplication map. We obtain a diagram of the form 
	\[
	\begin{tikzcd}
	K^{\otimes 2}\ar[r]\ar[d]&K\ar[d]\\
	\cA_{\Pro(E)}^{\otimes 2}\ar[r]&\cA_{\Pro(E)}
	\end{tikzcd}
	\]	which does a priori \textbf{not} commute, but where the vertical maps are $E_1$-isomorphisms and one obtains commutativity after applying $H_{BC}$ or $E_r$ for $r\geq 1$, thereby allowing to compute the product in cohomology.
\end{rem}
\section{Modifications}
The next goal is to compute the double complex of blow-ups up to $E_1$-isomorphism. Here is a general computation yielding a partial answer:

\begin{definition}
	For a map $p:X\longrightarrow Y$ between connected compact complex manifolds of complex dimensions $\dim X=n,\dim Y=m$, the \textbf{pushforward} $p_*$ is, with the notation of section \ref{sec: Preliminaries}, the composite
	\[\cA_X\overset{\Phi}{\longrightarrow}\mathcal{D}^{top}\cA_X\overset{\mathcal{D}p^*}{\longrightarrow}\mathcal{D}^{top}\cA_Y[n-m].\]
\end{definition}

\begin{lem}\label{complex of modification}
	For a surjective holomorphic map $f:Y\longrightarrow X$ of connected compact complex manifolds of the same  dimension, the map
	\[
	\cA_Y\overset{(f_*, \pr)}{\longrightarrow}\mathcal{D}^{top}\cA_X\oplus \cA_Y/f^*\cA_X
	\]
	is an $E_1$-isomorphism.
\end{lem}
\begin{proof}
	Since $f$ is a finite covering with $\deg(f)>0$ sheets when restricted to appropriate dense open subsets of $X$ and $Y$, {c.f. \cite[p. 179]{grauert_coherent_1984}}, one obtains an exact sequence
	\[\tag{$\ast$}
	0\longrightarrow\cA_X\overset{f^*}{\longrightarrow}\cA_Y\longrightarrow\cA_Y/f^*\cA_X\longrightarrow 0.
	\]
	As noted in \cite{wells_comparison_1974}, one has $\int_{Y}f^*\omega=\deg (f)\int_X\omega$ for any form $\omega$ on $X$, so that the diagram 
	\[
	\begin{tikzcd}
	\cA_X\ar[r,"f^*"]\ar[d,swap,"\deg(f)\cdot\Phi"]&\cA_{Y}\ar[d,"\Phi"]\\
	\mathcal{D}^{top}\cA_X&\mathcal{D}^{top}\cA_{Y}\ar[l,"\mathcal{D}f^*"]
	\end{tikzcd}
	\]
	commutes.
	Since $\Phi$ induces an isomorphism on the first page of the Fr\"olicher spectral sequence, for every $p\in\Z$, in the long exact sequence of terms on the first page of the Fr\"olicher spectral sequence induced by $(\ast)$
	\[
	...\overset{\delta}{\longrightarrow}H^{p,q}_{\delbar}(\cA_X)\overset{f^*}{\longrightarrow}H^{p,q}_{\delbar}(\cA_Y)\overset{\pr}{\longrightarrow}H^{p,q}_{\delbar}(\cA_Y/f^*\cA_X)\overset{\delta}{\longrightarrow}...
	\]
	the map $f^*$ is a split injection (with left inverse $\frac{1}{\deg(f)}f_*$) and hence $\pr$ is surjective. This implies that the morphism $(f_*,\pr)$ in the statement induces an isomorphism on the first page of the Fr\"olicher spectral sequence.
\end{proof}
Since $\Phi: \cA_X\longrightarrow \mathcal{D}^{top}\cA_X$ is an $E_1$-isomorphism, this implies in particular that the (Dolbeault, Bott-Chern, Aeppli or de Rham) cohomology of $X$ is a direct summand in that of $Y$. Thus, in order to compute the double complex of a blow-up, one just has to take care of the quotient-type summand. We will use this in the next section.

\section{Blow-ups}
\begin{thm}\label{complex of blow-up}
Let $X$ be a connected compact complex manifold, $Z\subsetneq X$ a closed submanifold and $\widetilde{X}$ the blow-up of $X$ at $Z$ and $E\subseteq \widetilde{X}$ the exceptional divisor, so that the following diagram is cartesian:
	\[
	\begin{tikzcd}
	E\ar[r,"j"]\ar[d,"\pi_E"]&\widetilde{X}\ar[d,"\pi"]\\
	Z\ar[r,"i"]&X
	\end{tikzcd}
	\]
The map
\[
\cA_{\widetilde{X}}\overset{(\pi_*,\overline{j^*})}{\longrightarrow}\mathcal{D}^{top}\cA_X\oplus\cA_{E}/\pi_E^*\cA_Z
\]
is an $E_1$-isomorphism, where $\overline{j^*}=\pr\circ j^*$.
\end{thm}

\begin{proof}
	By Lemma \ref{complex of modification}, it remains to show that the induced map 
	\[\widetilde{j^*}: \cA_{\tilde{X}}/ \pi^* \cA_X \longrightarrow \cA_E/ \pi^*_E \cA_Z\] is an $E_1$-isomorphism. This map sits inside a diagram
	\[
	\begin{tikzcd}
	0\ar[r]&\cA_X\ar[r,"\pi^*"]\ar[d,"i^*"]&\cA_{\widetilde{X}}\ar[r]\ar[d,"j^*"]&\cA_{\widetilde{X}}/\pi^*\cA_X\ar[r]\ar[d,"\widetilde{j^*}"]&0\\
	0\ar[r]&\cA_Z\ar[r,"\pi_E^*"]&\cA_{E}\ar[r]&\cA_{E}/\pi_E^*\cA_Z\ar[r]&0,
	\end{tikzcd}
	\]
	which, when applying Dolbeault cohomology, yields for every $p$ a map of two long exact sequences
	\[
	\begin{tikzcd}
	\cdots\ar[r]&H^{p,q}_{\delbar}(\cA_X)\ar[r,"\pi^*"]\ar[d,"i^*"]&H^{p,q}_{\delbar}(\cA_{\widetilde{X}})\ar[r]\ar[d,"j^*"]&H^{p,q}_{\delbar}(\cA_{\widetilde{X}}/\pi^*\cA_X)\ar[r]\ar[d,"\widetilde{j^*}"]&\cdots\\
	\cdots\ar[r]&H^{p,q}_{\delbar}(\cA_Z)\ar[r,"\pi_E^*"]&H^{p,q}_{\delbar}(\cA_{E})\ar[r]&H^{p,q}_{\delbar}(\cA_{E}/\pi_E^*\cA_Z\ar[r])&\cdots.
	\end{tikzcd}
	\]
But $\pi^*$ and $\pi_E^*$ in these sequences are injective: In fact, for $\pi^*$ this was shown in the proof of Lemma \ref{complex of modification} and since $E$ is a projective bundle over $Z$, for $\pi_E$ it follows from Proposition \ref{complex of a projective bundle}. In particular, the long exact sequences decompose into short exact ones and we obtain one diagram for every pair $(p,q)\in\Z^2$:
	\[
	\begin{tikzcd}\tag{$\ast\ast$} 
	0\ar[r]&H^{p,q}_{\delbar}(\cA_X)\ar[r,"\pi^*"]\ar[d,"i^*"]&H^{p,q}_{\delbar}(\cA_{\widetilde{X}})\ar[r]\ar[d,"j^*"]&H^{p,q}_{\delbar}(\cA_{\widetilde{X}}/\pi^*\cA_X)\ar[r]\ar[d,"\widetilde{j^*}"]&0\\
	0\ar[r]&H^{p,q}_{\delbar}(\cA_Z)\ar[r,"\pi_E^*"]&H^{p,q}_{\delbar}(\cA_{E})\ar[r]&H^{p,q}_{\delbar}(\cA_{E}/\pi_E^*\cA_Z\ar[r])&0.
	\end{tikzcd}
	\]

Now, note that $H^{p,q}(\cA_{\widetilde{X}})=H^{q}(\widetilde{X},\Omega_{\widetilde{X}}^p)$ is the cohomology of the sheaf of holomorphic $p$ forms (and similarly for $E$). Let us consider the Leray spectral sequences associated with $\pi$ and $\pi_E$ and the sheaves $\Omega_{\widetilde{X}}^p$ and $\Omega_E^p$:
\[
L_{\pi,\Omega_{\widetilde{X}}^p}:E_2^{r,s}=H^r(X,R^s\pi_*\Omega_{\widetilde{X}}^p)\Longrightarrow (H^{r+s}(\widetilde{X},\Omega_{\widetilde{X}}^p), F^\Cdot_L)
\]
and
\[
L_{\pi_E,\Omega_E^p}: E_2^{r,s}=H^r(Z,R^s{\pi_E}_*\Omega_{E}^p)\Longrightarrow (H^{r+s}(E,\Omega_{E}^p), F_L^\Cdot),
\]
where $F_L^\Cdot$ denotes in both cases the (descending) Leray-filtration on the target. Pullback by $j$ induces a morphism of spectral sequences
\[
j^*_L:L_{\pi,\Omega_{\widetilde{X}}^p}\longrightarrow L_{\pi_E,\Omega_E^p}
\]
which approximates $j^*$ on the target, i.e., on the $E_\infty$-page the $(r,s)$-component coincides with the $r$-th graded part of $j^*:H^{r+s}(\widetilde{X},\Omega_{\widetilde{X}}^p)\longrightarrow H^{r+s}(E,\Omega_E^p)$.\\
The sheaves $R^s\pi_*\Omega_{\widetilde{X}}^p$ and $R^s{\pi_E}_*\Omega_{E}^p$ have been investigated in \cite[Prop. 3.3]{guillen_critere_2002}, using a vanishing result by Bott and Serre's Theorem (B). The computation is done there for smooth schemes, but the result holds in the holomorphic category. In fact, this has already been deduced later in that same article \cite[p. 69]{guillen_critere_2002}, by local considerations. Alternatively, all tools used are also available in the holomorphic category, c.f. \cite[ch. IV]{banica_algebraic_1976}, and the proof can be copied verbatim. The result is that the following maps are isomorphisms for all $p\in\Z$:
\begin{align*}
\pi^*:&~ \Omega_X^p\cong \pi_*\Omega_{\widetilde{X}}^p&(i)\\
\pi_E^*:&~ \Omega_Z^p\cong {\pi_E}_*\Omega_E^p&(ii)\\
j^*:&~ R^s\pi_*\Omega_{\widetilde{X}}^p\cong i_*R^s{\pi_E}_*\Omega_{E}^p\text{ for each }s\geq 1.&(iii)
\end{align*}
The first two isomorphisms imply that the $E_2^{q,0}$-terms can be identified with $H^{p,q}_{\delbar}(\cA_X)$, resp. $H^{p,q}_{\delbar}(\cA_Z)$ and the edge maps with $\pi^*$, resp. $\pi_E^*$. Since these are injective, the diagram $(\ast\ast)$ is canonically isomorphic to
	\[
	\begin{tikzcd}
	0\ar[r]&F^q_LH^{p,q}_{\delbar}(\cA_{\widetilde{X}})\ar[r,"\subseteq"]\ar[d,"F_L^0j^*"]&H^{p,q}_{\delbar}(\cA_{\widetilde{X}})\ar[r]\ar[d,"j^*"]&H^{p,q}_{\delbar}(\cA_{\widetilde{X}})/F^q_LH^{p,q}_{\delbar}(\cA_{\widetilde{X}})\ar[r]\ar[d]&0\\
	0\ar[r]&F_L^qH^{p,q}_{\delbar}(\cA_{E})\ar[r,"\subseteq"]&H^{p,q}_{\delbar}(\cA_{E})\ar[r]&H^{p,q}_{\delbar}(\cA_{E})/F^q_LH^{p,q}_{\delbar}(\cA_{E})\ar[r]&0
	\end{tikzcd}
	\]
Finally, all differentials with target in degree $(r,0)$ for some $r\in\Z$ vanish since the edge maps are injective and the identification $(iii)$ implies that $j^*_L$ is an isomorphism on the $E_2$-page in bidegrees $(r,s)$ for all $r\in\Z$, $s\geq 1$. Therefore, $j^*$ induces isomorphisms 
\[
j^*:\operatorname{gr}_{F_L}^r\! H^{p,q}_{\delbar}(\cA_{\widetilde{X}})\cong \operatorname{gr}_{F_L}^r\!H^{p,q}_{\delbar}(\cA_E)
\]
for all $p,q\in\Z$ and $r<q$. In particular, since a filtered morphism of vector spaces with finite filtrations is an isomorphism if its associated graded is, $\widetilde{j^*}$ is an isomorphism in Dolbeault cohomology.
\end{proof}
\begin{rem}
	(Inverse map and products) The morphism of the theorem fits into a diagram of the form
	\[
	\begin{tikzcd}
	\cA_{\widetilde{X}}\ar[r]\ar[d]&\mathcal{D}^{top}\cA_X\oplus\cA_{E}/\pi_E^*\cA_Z\\
	\mathcal{D}^{top}\cA_{\widetilde{X}}&\cA_X\oplus\bigoplus_{i=0}^{r-2}\cA_Z[i+1]\ar[l]\ar[u]
	\end{tikzcd}
	\]
where the bottom map is given by 
\[(\omega,\eta_0,...,\eta_{r-2})\longmapsto \pi^*\omega+\sum_{i=0}^{r-2}j_*(\Phi\circ\pi_E^*\eta_i\wedge\theta^i)\]
and the right vertical map by $\Phi\oplus\sum_{i=0}^{r-2}(\pi_E^*(\_)\wedge\theta^{i+1})$. As we have seen in the present article, the top map and the vertical ones are $E_1$-isomorphisms. On the other hand, the results of \cite{meng_explicit_2018} and \cite{meng_mayer-vietoris_2018} (c.f. also \cite{angella_note_2017}) imply that the bottom map is an $E_1$-isomorphism. It seems to be likely, but as far as I know unproven at the present stage, that the diagram commutes on $E_1$ (say, after replacing one of the horizontal arrows by its inverse in cohomology). This is further discussed in \cite[section. 6]{meng_mayer-vietoris_2018}. If this was true, it should in particular allow to describe the induced product on $H(X)\oplus\bigoplus_{i=0}^{r-2}H(X)\wedge [\theta^i]$ explicitely (for every $H$ where this is meanigful).

\end{rem}

Let us give a concrete example of how one may compute with the formulas in this article:

\begin{ex} Let $X$ be the Iwasawa manifold, i.e. $X$ is the space of complex upper triangular $3\times 3$ matrices with $1$'s on the diagonal modulo the lattice of such matrices with values in the Gaussian integers. This is a complex non-K\"ahler manifold that has been extensively studied (see e.g. \cite{angella_cohomologies_2013-1}, which also contains further references). The Fr\"olicher spectral sequence of $X$ degenerates at the second page and the dimensions of $E_1,E_2$, $H_{dR}(X)$ and $H_{BC}(X)$ are given as:

\tiny
\[
\begin{array}{ccccccc}
\begin{array}{ccccccc}
&&&1&&&\\
&&3&&2&&\\
&3&&6&&2&\\
1&&6&&6&&1\\
&2&&6&&3&\\
&&2&&3&&\\
&&&1&&&
\end{array}
&&
\begin{array}{ccccccc}
&&&1&&&\\
&&2&&2&&\\
&2&&4&&2&\\
1&&4&&4&&1\\
&2&&4&&2&\\
&&2&&2&&\\
&&&1&&&
\end{array}
&&
\begin{array}{ccccccc}
&&&1&&&\\
&&3&&3&&\\
&2&&8&&2&\\
1&&6&&6&&1\\
&3&&4&&3&\\
&&2&&2&&\\
&&&1&&&
\end{array}
&&
\begin{array}{c}
1\\
4\\
8\\
10\\
8\\
4\\
1
\end{array}\\
\dim E_1^{p,q}(X) && \dim E_2^{p,q}(X) && \dim H_{BC}(X)^{p,q} && b_k(X)
\end{array}
\]
\normalsize

$X$ is a fibre bundle over a complex $2$-dimensional torus with fibre a complex $1$-dimensional torus. Let $Z$ be any one such torus fibre. The blow-up $\widetilde{X}$ of $X$ along $Z$ then one has the following $E_1$-isomorphisms:
\begin{align*}
\cA_{\widetilde{X}}\overset{\sim}{\longrightarrow}\mathcal{D}^{top}\cA_{\widetilde X}\oplus \cA_{E}/\pi_E^*\cA_Z
\overset{\simeq}{\longleftarrow}\cA_X\oplus\cA_Z[1].
\end{align*}
Since $Z$ is a torus (in particular, K\"ahler), one has $\dim H_{BC}^{p,q}(X)=\dim H_{\delbar}^{p,q}(Z)=1$ for $p,q\in\{0,1\}$ and $E_1(Z)= E_\infty(Z)$ and the new cohomology groups have dimensions:
\tiny
\[
\begin{array}{ccccccc}
\begin{array}{ccccccc}
&&&1&&&\\
&&3&&2&&\\
&3&&7&&2&\\
1&&7&&7&&1\\
&2&&7&&3&\\
&&2&&3&&\\
&&&1&&&
\end{array}
&&
\begin{array}{ccccccc}
&&&1&&&\\
&&2&&2&&\\
&2&&5&&2&\\
1&&5&&5&&1\\
&2&&5&&2&\\
&&2&&2&&\\
&&&1&&&
\end{array}
&&
\begin{array}{ccccccc}
&&&1&&&\\
&&3&&3&&\\
&2&&9&&2&\\
1&&7&&7&&1\\
&3&&5&&3&\\
&&2&&2&&\\
&&&1&&&
\end{array}
&&
\begin{array}{c}
1\\
4\\
9\\
12\\
9\\
4\\
1
\end{array}\\
\dim E_1^{p,q}(\widetilde{X}) && \dim E_2^{p,q}(\widetilde{X}) && \dim H_{BC}(\widetilde{X})^{p,q} && b_k(\widetilde{X})
\end{array}
\]
\normalsize

The reader looking for more involved examples could now compute the dimensions of cohomology vector spaces of a projective bundle over $X$ or of the blow-up of $X\times \Pro^n$ along $X\times\Pro^m$ for $m\leq n-2$ by essentially the same techniques.
\end{ex}

\begin{rem}
	In section $4.$ of \cite{guillen_critere_2002}, the results proven earlier in that article are used to extend the theory of holomorphic de-Rham complex and the Hodge filtration to possibly singular complex analytic spaces, using resolutions of singularities by iterated blow-ups.\\
	The results of the present article may be read as a formal equality of $E_1$-isomorphism classes 
	\[
	[\cA_{\widetilde{X}}\oplus \cA_Z]=[\cA_X\oplus\cA_E].
	\]
	If one allows formal additive inverses of $E_1$-isomorphism classes, one may use this relation to associate with every possibly singular compact (or even compactifiable with fixed equivalence class of compactifications) complex space a ``virtual Dolbeault double complex'', i.e., an element in the Grothendieck group of the monoid consiting of double complexes with finite $E_1$-page and direct sum as addition. Roughly, this may be achieved by replacing $X$ with a diagram consisting of compact complex manifolds using resolution of singularities and the weak factorization theorem. In particular, this allows to define the Hodge, $E_r$-, Bott-Chern and Aeppli numbers for singular spaces. Of course, for the Hodge (and $E_r$-) numbers, this can already be achieved by the results of \cite[sect. 4]{guillen_critere_2002}. Instead of giving more details, we refer to \cite{peters_hodge_2007} where an additive extension of the Hodge numbers to singular spaces is constructed in the spirit sketched here, using a result in \cite{bittner_universal_2004}. 
\end{rem}

\textbf{Acknowledgements:} The contents of this article were mostly elaborated with financial support by the SFB 878 at the WWU M\"unster and profited from many helpful comments and advice from Christopher Deninger and J\"org Sch\"urmann. I gratefully acknowledge this support. In particular, the results in this text were prompted by a question from J\"org Sch\"urmann, who also suggested that the results in \cite{guillen_critere_2002} might be useful. I would also like to thank Daniele Angella for encouraging discussions on the subject, Tatsuo Suwa for pointing out a gap in an early version of the article and Sheng Rao, Song Yang and Xiangdong Yang for their detailed reading of the preprint version. Finally, I thank the anonymous referees for many interesting suggestions and remarks.


\begin{thebibliography}{99}

\bibitem[\protect\citeauthoryear{Angella}{Angella}{2013a}]{angella_cohomological_2013}
Angella, D. (2013a, November).
\newblock {\em Cohomological {{Aspects}} in {{Complex Non}}-{{K\"ahler
  Geometry}}}.
\newblock Number 2095 in Lecture {{Notes}} in {{Math}}. {Springer}.

\bibitem[\protect\citeauthoryear{Angella}{Angella}{2013b}]{angella_cohomologies_2013-1}
Angella, D. (2013b, July).
\newblock The cohomologies of the {{Iwasawa}} manifold and of its small
  deformations.
\newblock {\em J. Geom. Anal.\/}~{\em 23\/}(3), 1355--1378.

\bibitem[\protect\citeauthoryear{Angella and Kasuya}{Angella and
  Kasuya}{2017}]{angella_cohomologies_2017}
Angella, D. and H.~Kasuya (2017).
\newblock Cohomologies of deformations of solvmanifolds and closedness of some
  properties.
\newblock {\em North-West. Eur. J. Math. 3\/}.

\bibitem[\protect\citeauthoryear{Angella, Suwa, Tardini, and Tomassini}{Angella
  et~al.}{2019}]{angella_note_2017}
Angella, D., T.~Suwa, N.~Tardini, and A.~Tomassini (2019, March).
\newblock Note on dolbeault cohomology and hodge structures up to
  bimeromorphisms.
\newblock {\em arXiv:1712.08889v2 [math]\/}.

\bibitem[\protect\citeauthoryear{B{\u a}nic{\u a} and St{\u a}n{\u a}\c sil{\u
  a}}{B{\u a}nic{\u a} and St{\u a}n{\u a}\c sil{\u
  a}}{1976}]{banica_algebraic_1976}
B{\u a}nic{\u a}, C. and O.~St{\u a}n{\u a}\c sil{\u a} (1976).
\newblock {\em Algebraic Methods in the Global Theory of Complex Spaces}.
\newblock Bucure\c sti: {Editura Academiei}.

\bibitem[\protect\citeauthoryear{Bittner}{Bittner}{2004}]{bittner_universal_2004}
Bittner, F. (2004, July).
\newblock The universal {{Euler}} characteristic for varieties of
  characteristic zero.
\newblock {\em Compos. Math.\/}~{\em 140\/}(04), 1011--1032.

\bibitem[\protect\citeauthoryear{Cordero, Fernandez, Gray, and Ugarte}{Cordero
  et~al.}{2000}]{cordero_compact_2000}
Cordero, L.~A., M.~Fernandez, A.~Gray, and L.~Ugarte (2000).
\newblock Compact nilmanifolds with nilpotent complex structures: Dolbeault
  cohomology.
\newblock {\em Trans. Amer. Math. Soc.\/}~{\em
  352\/}(12), 5405--5433.

\bibitem[\protect\citeauthoryear{Grauert and Remmert}{Grauert and
  Remmert}{1984}]{grauert_coherent_1984}
Grauert, H. and R.~Remmert (1984, July).
\newblock {\em Coherent {{Analytic Sheaves}}},  Volume 131 of {\em
  Grundlehren Math. Wiss}.
\newblock  {Springer}.

\bibitem[\protect\citeauthoryear{Griffiths and Harris}{Griffiths and
  Harris}{1978}]{griffiths_principles_1978}
Griffiths, P. and J.~Harris (1978, October).
\newblock {\em Principles of {{Algebraic Geometry}}}.
\newblock {Wiley}.

\bibitem[\protect\citeauthoryear{Guill\'en and Navarro~Aznar}{Guill\'en and
  Navarro~Aznar}{2002}]{guillen_critere_2002}
Guill\'en, F. and V.~Navarro~Aznar (2002, July).
\newblock {Un crit\`ere d'extension des foncteurs d\'efinis sur les sch\'emas
  lisses}.
\newblock {\em Publ. Math. Inst. Hautes \'Etudes Sci.\/}~{\em 95\/}(1),
  1--91.

\bibitem[\protect\citeauthoryear{Hirzebruch}{Hirzebruch}{1978}]{hirzebruch_topological_1978}
Hirzebruch, F. (1978).
\newblock {\em Topological Methods in Algebraic Geometry}, Volume 131 of {\em
  Grundlehren Math. Wiss}.
\newblock {Springer}.

\bibitem[\protect\citeauthoryear{Meng}{Meng}{2018a}]{meng_explicit_2018}
Meng, L. (2018a, October).
\newblock An explicit formula of blow-ups for {{Dolbeault}} cohomology.
\newblock {\em arXiv:1806.11435v4 [math]\/}.

\bibitem[\protect\citeauthoryear{Meng}{Meng}{2018b}]{meng_morse-novikov_2018}
Meng, L. (2018b, October).
\newblock Morse-{{Novikov Cohomology}} for {{Blow}}-ups of {{Complex
  Manifolds}}.
\newblock {\em arXiv:1806.06622v3 [math]\/}.

\bibitem[\protect\citeauthoryear{Meng}{Meng}{2019}]{meng_mayer-vietoris_2018}
Meng, L. (2019, April).
\newblock Mayer-{{Vietoris}} systems and their applications.
\newblock {\em arXiv:1811.10500v3 [math]\/}.



\bibitem[\protect\citeauthoryear{Peters and Steenbrink}{Peters and
  Steenbrink}{2007}]{peters_hodge_2007}
Peters, C. a.~M. and J.~H.~M. Steenbrink (2007).
\newblock Hodge number polynomials for nearby and vanishing cohomology.
\newblock {\em Algebraic Cycles and Motives\/}, 289--303.

\bibitem[\protect\citeauthoryear{Rao, Yang, and Yang}{Rao
  et~al.}{2018}]{rao_dolbeault_2018}
Rao, S., S.~Yang, and X.~Yang (2018, December).
\newblock Dolbeault cohomologies of blowing up complex manifolds {{II}}:
  bundle-valued case.
\newblock {\em arXiv:1809.07277v2 [math]\/}.

\bibitem[\protect\citeauthoryear{Rao, Yang, and Yang}{Rao
  et~al.}{2019}]{rao_dolbeault_2019}
Rao, S., S.~Yang, and X.~Yang (2019, January).
\newblock Dolbeault cohomologies of blowing up complex manifolds.
\newblock {\em J. Math. Pures Appl.\/}.

\bibitem[\protect\citeauthoryear{Serre}{Serre}{1955}]{serre_theoreme_1955}
Serre, J.-P. (1955).
\newblock {Un th\'eor\`eme de dualit\'e.}
\newblock {\em Comment. Math. Helv.\/}~{\em 29}, 9--26.

\bibitem[\protect\citeauthoryear{Stelzig}{Stelzig}{2018}]{stelzig_structure_2018}
Stelzig, J. (2018, December).
\newblock On the {{Structure}} of {{Double Complexes}}.
\newblock {\em arXiv:1812.00865v1 [math]\/}.

\bibitem[\protect\citeauthoryear{Voisin}{Voisin}{2002}]{voisin_hodge_2002}
Voisin, C. (2002, December).
\newblock {\em Hodge {{Theory}} and {{Complex Algebraic Geometry I}}:}.
\newblock {Cambridge University Press}.

\bibitem[\protect\citeauthoryear{Wells}{Wells}{1974}]{wells_comparison_1974}
Wells, R. (1974).
\newblock Comparison of de {{Rham}} and {{Dolbeault}} cohomology for proper
  surjective mappings.
\newblock {\em Pacific J. Math.\/}~{\em 53\/}(1), 281--300.

\bibitem[\protect\citeauthoryear{Yang and Yang}{Yang and
  Yang}{2018}]{yang_bott-chern_2017}
Yang, S. and X.~Yang (2018, September).
\newblock Bott-chern blow-up formula and bimeromorphic invariance of the
  $\partial\bar{\partial}$-lemma for threefolds.
\newblock {\em arXiv:1712.08901v3 [math]\/}.

\end{thebibliography}
\end{document}